\documentclass[a4paper,reqno,11pt]{amsart}
\usepackage{amsmath,amssymb}
\usepackage[top=25mm,right=25mm,bottom=25mm,left=25mm]{geometry}
\usepackage[hypertexnames=false,hidelinks]{hyperref}

\newcommand\F{\mathbb{F}}
\newcommand\GL{\operatorname{GL}}
\newcommand\AGL{\operatorname{AGL}}
\newcommand\diam{\operatorname{diam}}
\newcommand\Symgrp{\operatorname{Sym}}

\newtheorem{thm}{Theorem}[section]
\newtheorem{prop}[thm]{Proposition}
\newtheorem{lem}[thm]{Lemma}
\newtheorem{conj}[thm]{Conjecture}
\newtheorem{ques}[thm]{Question}
\numberwithin{equation}{section}

\begin{document}

\title[Common neighbour conjectures for Saxl graphs fail at every base size]{Common neighbour conjectures for Saxl graphs fail at every base size}

\author[Aluna Rizzoli]{Aluna Rizzoli}
\address{Department of Mathematics, King's College London, Strand, London WC2R 2LS, United Kingdom}
\address{Heilbronn Institute for Mathematical Research, Bristol, United Kingdom}
\email{aluna.rizzoli@kcl.ac.uk}

\author[Adam R. Thomas]{Adam R. Thomas}
\address{Warwick Mathematics Institute, University of Warwick, Coventry CV4 7AL, United Kingdom}
\email{adam.r.thomas@warwick.ac.uk}

\begin{abstract} For a finite permutation group, a base is a set of points with trivial pointwise stabiliser, and the generalised Saxl graph records which pairs of points lie together in a base of minimum size. Burness and Giudici conjectured that any two vertices of the Saxl graph of a primitive group of base size two have a common neighbour, and Freedman, Huang, Lee and Rekv\'enyi extended this conjecture to arbitrary base size. We disprove both. For each integer $B\ge2$ we construct infinitely many primitive groups of base size $B$ whose generalised Saxl graphs contain two nonadjacent vertices with no common neighbour. At base size two, where this is the usual Saxl graph, we obtain three further infinite families, one each of affine, product and twisted wreath type, so the conjecture fails in three of the five O'Nan--Scott types; in the affine and product type families the Saxl graphs have diameter exactly three. This answers Problem~21.29 in the Kourovka Notebook in the negative. In the positive direction, we prove the Burness--Giudici conjecture for every primitive affine group whose point stabiliser is almost quasisimple of sporadic type, completing work of Lee and Popiel. We conjecture that no base-two counterexample of almost simple or diagonal type exists. \end{abstract}

\maketitle
\vspace{-26pt}

\section{Introduction}

Let $G\le\Symgrp(\Omega)$ be a finite transitive permutation group. A \emph{base} for $G$ is a subset of $\Omega$ with trivial pointwise stabiliser, and the \emph{base size} $b(G)$ is the least size of a base. Bases are a classical theme in permutation group theory, with a history going back to the nineteenth century and a central role in computational group theory; see the survey \cite{BC} and the introductions of \cite{BG,FHLR} for background. Much recent work concerns primitive groups with $b(G)=2$, the least possible base size for a nonregular group.

To study these groups, Burness and Giudici \cite{BG} introduced the \emph{Saxl graph} $\Sigma(G)$: its vertices are the points of $\Omega$, and two vertices are adjacent if they form a base. The graph encodes the base-two structure of $G$. They develop its basic theory in \cite{BG}, studying in particular its connectivity, diameter and valency. If $G$ is primitive then $\Sigma(G)$ is connected \cite[Lemma~2.1(ii)]{BG}. At the heart of that paper is the following conjecture \cite[Conjecture~4.5]{BG}.

\begin{conj}[Common Neighbour Conjecture]\label{bgconj}
Let $G$ be a primitive permutation group with $b(G)=2$. Then any two vertices of $\Sigma(G)$ have a common neighbour.
\end{conj}

In particular, Conjecture~\ref{bgconj} implies that $\Sigma(G)$ has diameter at most $2$. Freedman, Huang, Lee and Rekv\'enyi \cite{FHLR} extended the definition of the Saxl graph to arbitrary base size. The \emph{generalised Saxl graph} of a group $G$ with $b(G)\ge2$ joins two points when they lie in a common base of size $b(G)$. For primitive $G$ they conjecture that any two vertices have a common neighbour \cite[Conjecture~1.2]{FHLR}. They proved this for several families, including almost simple primitive groups with soluble point stabilisers.

This paper shows that both conjectures are false, yielding a negative answer to Problem~21.29 in the Kourovka Notebook \cite{KN}, which asks precisely if Conjecture~\ref{bgconj} holds.

\begin{thm}\label{mainthm}
For every integer $B\ge2$ there exist primitive permutation groups $G$ of arbitrarily large degree with $b(G)=B$ whose generalised Saxl graphs contain two nonadjacent vertices with no common neighbour.
\end{thm}

We prove Theorem~\ref{mainthm} using the affine family of Section~\ref{sec:everybase}. The construction is defined for every base size $B\ge2$, and taking $B=2$ gives an infinite family of affine counterexamples to Conjecture~\ref{bgconj}. We have formally verified Theorem~\ref{mainthm} in Lean~4 \cite{Lean4} using this construction. The statement of the formalised theorem relies on only four definitions beyond Mathlib \cite{Mathlib}: a finite faithful permutation group, a base, base size and the generalised Saxl graph. The proof contains no \texttt{sorry} placeholders or project-specific axioms, and the Lean source is in the same repository~\cite{AlunaAdam}.

In Sections~\ref{sec:affine}--\ref{sec:tw} we construct three further infinite families of base-two primitive groups, one each of affine, product and twisted wreath type. The failure of the conjecture therefore occurs in three of the five O'Nan--Scott types \cite[Theorem~4.1A]{DM}. In the affine and product type families, every counterexample has Saxl-graph diameter exactly $3$. For each of the three families, the analysis relies on explicit computational input (Lemmas~\ref{affineinput}, \ref{productinput} and~\ref{twinput}) and on the machinery of Section~\ref{sec:prelim}. Self-contained GAP \cite{GAP} scripts performing these computations are available in the repository \cite{AlunaAdam}.

 Writing $\Sigma_G(\alpha)$ for the neighbourhood of $\alpha$, Burness and Huang proposed the stronger condition
\[
(\dagger)\qquad
\Sigma_G(\alpha)\text{ meets every regular }G_\beta\text{-orbit on }\Omega
\quad\text{for all }\alpha,\beta\in\Omega.
\]
Thus $(\dagger)$ implies that any two vertices have at least $r(G)$ common neighbours, where $r(G)$ is the number of regular suborbits. Although this condition is genuinely stronger for an individual action, Burness and Huang proved that Conjecture~\ref{bgconj} is equivalent to the universal statement that every primitive base-two group satisfies $(\dagger)$ \cite[Conjecture~5.8 and Proposition~5.9]{BH}.

Before the present work, substantial positive evidence had accumulated. Burness and Giudici verified Conjecture~\ref{bgconj} for various families of almost simple groups \cite{BG}, while Burness and Huang verified $(\dagger)$ computationally for every primitive group of degree at most $4095$ \cite[Remark~5.10]{BH}. The common neighbour conjecture was proved by Burness and Huang for every almost simple primitive group with soluble point stabilisers \cite{BHs}. Chen and Du proved the weaker diameter-two statement for groups with socle $\operatorname{PSL}_2(q)$ \cite{CD}, while Burness and Huang established the common neighbour conjecture in this case \cite[Theorem~4.22]{BHs}. The remaining rank-one socles were treated in recent preprints, $\operatorname{PSU}_3(q)$ in \cite{CDL}, $\operatorname{Ree}(q)$ and $\operatorname{Sz}(q)$ in \cite{CDR}. 

Recall that a group $H$ is \emph{almost quasisimple} if it has a unique quasisimple subnormal subgroup. We say that such a group is \emph{of sporadic type} if $\operatorname{soc}(H/Z(H))$ is a sporadic simple group. Lee and Popiel proved Conjecture~\ref{bgconj} for affine groups whose point stabiliser is almost quasisimple of sporadic type, apart from ten cases \cite{LP}. In Section~\ref{sec:searches} we verify the conjecture in these ten cases. Huang proved the common neighbour conjecture for diagonal type groups whose top group is neither $A_k$ nor $S_k$ \cite[Theorem~5.6]{HuT} (see also \cite{Hu}).

The point stabilisers in the product type family are soluble, so the soluble-stabiliser theorem of \cite{BHs} does not extend from almost simple groups to product type. At the other extreme, Proposition~\ref{prop:perfect} gives an affine counterexample of degree $3^{15}$ whose point stabiliser $\AGL_4(2)$ is perfect. Both non-affine families have a unique regular suborbit, so for them $(\dagger)$ reduces to the ordinary common neighbour condition.

Complementing these results, the searches of Section~\ref{sec:searches} found no counterexample of non-affine type of degree at most $10^8$, and none of diagonal type of degree at most $10^{24}$ (Table~\ref{tab:ranges}). Since Sections~\ref{sec:affine}--\ref{sec:tw} provide counterexamples of affine, product and twisted wreath type, only the almost simple and diagonal types remain open. Together with the results cited above, this suggests the following amended conjecture.

\begin{conj}\label{ourconj}
Let $G$ be a primitive permutation group of almost simple or diagonal type with $b(G)=2$. Then any two vertices of $\Sigma(G)$ have a common neighbour.
\end{conj}

We know of no Saxl graph of a base-two primitive permutation group $G$ of diameter greater than $3$. All base-two primitive groups of degree at most $2^{18}-1$ have Saxl-graph diameter at most $3$, as do the groups in the affine family of Section~\ref{sec:affine} and the product type family. For the twisted wreath family we only know that the diameter is at least $3$. We therefore ask the following.

\begin{ques}\label{diamq}
Is the diameter of the Saxl graph of a primitive permutation group $G$ with $b(G)=2$ always at most $3$?
\end{ques}

\section{Preliminaries} \label{sec:prelim}

This section introduces the product action machinery used in Sections~\ref{sec:affine}--\ref{sec:everybase} to produce infinite families of counterexamples.

Let $X\le\Symgrp(\Delta)$ be a transitive group with nontrivial point stabilisers. For $\alpha\in\Delta$, the orbits of $X_\alpha$ on $\Delta$ are the \emph{suborbits} of $X$, and such an orbit is \emph{regular} if $X_\alpha$ acts regularly on it. By transitivity the number $r(X)$ of regular suborbits does not depend on $\alpha$. An orbit of $X$ on $\Delta^2$ is an \emph{orbital}. Write $\mathcal R(X)$ for the set of \emph{regular orbitals} of $X$, namely the $X$-orbits on the \emph{ordered bases} of size two, or equivalently the orbitals on which $X$ acts regularly. Choosing a point gives a bijection between $\mathcal R(X)$ and the regular suborbits, so $|\mathcal R(X)|=r(X)$; see \cite[Remark~2.2]{BG}. We say that $\mathcal O\in\mathcal R(X)$ is \emph{self-paired} if $\mathcal O=\{(\beta,\alpha)\mid(\alpha,\beta)\in\mathcal O\}$.

A partition $\Pi$ of $\{1,\dots,m\}$ is \emph{distinguishing} for a transitive $Q\le S_m$ if the intersection of the setwise stabilisers in $Q$ of its parts is trivial, and the \emph{distinguishing number} $D(Q)$ is the least number of parts in such a partition; see \cite[Section~2]{BG}. If $C_\ell$ has prime order $\ell$ and acts regularly, every partition with at least two parts is distinguishing, so $D(C_\ell)=2$. For the natural action of $S_m$, only the partition into singletons is distinguishing, so $D(S_m)=m$.

Now let $G=X\wr Q$ act on $\Delta^m$ in product action. For $a\in\Delta$, write $a^m=(a,\dots,a)$ for the constant $m$-tuple. We record the standard primitivity criterion for this action \cite[Lemma~2.7A]{DM}.

\begin{lem}\label{paprim}
Let $X\le\Symgrp(\Delta)$ and $Q\le S_m$ be nontrivial, where $m\ge2$. Then $X\wr Q$ is primitive on $\Delta^m$ in product action if and only if $X$ is primitive but not regular on $\Delta$ and $Q$ is transitive on $\{1,\dots,m\}$.
\end{lem}

Let $x=(x_1,\dots,x_m)$ and $y=(y_1,\dots,y_m)$ in $\Delta^m$ be such that each $(x_j,y_j)$ is an ordered base for $X$, and define the tuple of coordinate orbitals
\[
\rho_{x,y}=\bigl((x_1,y_1)^X,\dots,(x_m,y_m)^X\bigr)
\in\mathcal R(X)^m.
\]
For $\rho=(\rho_1,\dots,\rho_m)\in\mathcal R(X)^m$, write $\Pi(\rho)$ for the partition of $\{1,\dots,m\}$ in which $i$ and $j$ lie in the same part if and only if $\rho_i=\rho_j$. We use the following standard criterion; see \cite[Lemma~2.8]{BG} and \cite[Lemma~4.1]{BH}.

\begin{lem}\label{basecrit}
Let $x,y\in\Delta^m$. Then $\{x,y\}$ is a base for $G=X\wr Q$ if and only if each $(x_j,y_j)$ is an ordered base for $X$ and the partition $\Pi(\rho_{x,y})$ is distinguishing for $Q$.
\end{lem}

The base size and the number of regular suborbits of such a wreath product are likewise known. The following combines \cite[Corollary~2.9]{BG} with \cite[Theorem~4.4, Remark~4.5 and Corollary~4.7]{BH}.

\begin{prop}\label{bookkeeping}
Let $X$ be transitive with $b(X)=2$ and let $Q\le S_m$ be transitive. Then the map induced by $(x,y)\mapsto\rho_{x,y}$ identifies the elements of $\mathcal R(X\wr Q)$ with the $Q$-orbits of tuples $\rho\in\mathcal R(X)^m$ for which $\Pi(\rho)$ is distinguishing. Consequently, $b(X\wr Q)=2$ if and only if $r(X)\ge D(Q)$. In particular, for $C_\ell$ of prime order $\ell$ acting regularly,
\begin{equation}\label{eq:cyc}
r(X\wr C_\ell)=\frac{r(X)^\ell-r(X)}{\ell}.
\end{equation}
\end{prop}

The following property passes to product action wreath products.

\begin{lem}\label{inherit}
Let $X$ be transitive on $\Delta$ with $b(X)=2$, let $Q\le S_m$ be transitive with $m\ge2$, and suppose $b(W)=2$ for $W=X\wr Q$. Suppose that, for every $\mathcal O\in\mathcal R(X)$ and every $\alpha,\beta\in\Delta$, there exist $\gamma,\delta\in\Delta$ such that $(\alpha,\gamma),(\gamma,\delta),(\delta,\beta)\in\mathcal O$. Then, for every $\mathcal P\in\mathcal R(W)$ and every $x,y\in\Delta^m$, there exist $u,v\in\Delta^m$ such that $(x,u),(u,v),(v,y)\in\mathcal P$. In particular, $\diam\Sigma(W)\le 3$.
\end{lem}

\begin{proof}
Let $\mathcal P\in\mathcal R(W)$ be represented by $\rho=(\mathcal O_1,\dots,\mathcal O_m)\in\mathcal R(X)^m$, where $\Pi(\rho)$ is distinguishing. Given $x,y\in\Delta^m$, choose, for each $j$, points $u_j,v_j\in\Delta$ such that $(x_j,u_j),(u_j,v_j),(v_j,y_j)\in\mathcal O_j$. With $u=(u_j)$ and $v=(v_j)$ we then have $\rho_{x,u}=\rho_{u,v}=\rho_{v,y}=\rho$, so each of $(x,u)$, $(u,v)$ and $(v,y)$ belongs to $\mathcal P$. Since $b(W)=2$, the set $\mathcal R(W)$ is nonempty, and the diameter bound follows.
\end{proof}

For a vertex $a$ of $\Sigma(X)$, recall that $\Sigma_X(a)$ denotes its neighbourhood. The next proposition starts from a particularly simple failure of $(\dagger)$. When $r(X)=2$, its hypothesis says precisely that the pair $(\alpha,\beta)$ witnesses such a failure. The proof is built on the argument for constant tuples in \cite[Lemma~5.6]{BH}.

\begin{prop}\label{cprop}
Let $X$ be transitive on $\Delta$ with $b(X)=2$ and $r(X)\ge 2$, and let $\alpha,\beta\in\Delta$ be distinct points such that $\Sigma_X(\alpha)$ meets at most one regular $X_\beta$-orbit. Let $\ell$ be prime and let $W=X\wr C_\ell$ act on $\Delta^\ell$ in product action, with $C_\ell$ regular. Then $b(W)=2$, and $\alpha^\ell$ and $\beta^\ell$ are nonadjacent vertices of $\Sigma(W)$ with no common neighbour.
\end{prop}

\begin{proof}
Since $D(C_\ell)=2$, Proposition~\ref{bookkeeping} gives $b(W)=2$. If $\{\alpha,\beta\}$ is not a base for $X$, then $\{\alpha^\ell,\beta^\ell\}$ is not a base for $W$ by Lemma~\ref{basecrit}. If it is, then $\rho_{\alpha^\ell,\beta^\ell}$ is constant, so its associated partition has one part and is not distinguishing. In either case, $\alpha^\ell$ and $\beta^\ell$ are nonadjacent.

Suppose that $z=(z_1,\dots,z_\ell)$ is a common neighbour of $\alpha^\ell$ and $\beta^\ell$. Each $z_j$ lies in $\Sigma_X(\alpha)$ and in a regular $X_\beta$-orbit. Moreover, Lemma~\ref{basecrit} shows that $\Pi(\rho_{\beta^\ell,z})$ has at least $D(C_\ell)=2$ parts. Therefore, since the regular orbitals of $X$ correspond to the regular $X_\beta$-orbits, the $z_j$ lie in at least two distinct such orbits, contrary to the hypothesis.
\end{proof}

The next result starts from a pair $\alpha,\beta$ such that, for every common neighbour $\gamma$, the points $\alpha$ and $\beta$ lie in different regular $X_\gamma$-orbits. This property is preserved on passing to a cyclic wreath product and, after a final wreath product with a symmetric group, produces a pair with no common neighbour.

\begin{prop}\label{sprop}
Let $X$ be transitive on $\Delta$ with $b(X)=2$ and $r(X)\ge2$. Let $\alpha,\beta\in\Delta$ be distinct points such that no common neighbour $\gamma$ of $\alpha$ and $\beta$ satisfies $(\gamma,\alpha)^X=(\gamma,\beta)^X$. Let $Y$ be obtained from $X$ by finitely many, possibly zero, product action wreath products with regular cyclic groups of prime order. Suppose that every group formed along the way, including $Y$ itself, has at least two regular orbitals. Let $W=Y\wr S_{r(Y)}$ act in product action. Then $b(W)=2$, $r(W)=1$, and $\Sigma(W)$ contains two nonadjacent vertices with no common neighbour.
\end{prop}

\begin{proof}
If some $x\in\Delta$ is such that $(x,\alpha)$ and $(x,\beta)$ are ordered bases lying in the same $X$-orbital, then $x$ is a common neighbour of $\alpha$ and $\beta$ with $(x,\alpha)^X=(x,\beta)^X$, contrary to hypothesis. Both steps below end by producing such a point. Each also uses the elementary fact that two tuples are equal if they have the same multiset of entries and agree in all but one coordinate.

We first show that the hypothesis passes up one cyclic layer. Let $\ell$ be prime and let $Y_1=X\wr C_\ell$ act on $\Delta^\ell$ in product action. Proposition~\ref{bookkeeping} gives $b(Y_1)=2$. Suppose that $\gamma$ is a common neighbour of $\alpha^\ell$ and $\beta'=(\alpha^{\ell-1},\beta)$ with $(\gamma,\alpha^\ell)$ and $(\gamma,\beta')$ in the same $Y_1$-orbital. The tuples of coordinate orbitals $u=\bigl((\gamma_j,\alpha)^X\bigr)_j$ and $v=\bigl((\gamma_j,\beta'_j)^X\bigr)_j$ agree in the first $\ell-1$ entries, and Proposition~\ref{bookkeeping} shows that the equality of the two $Y_1$-orbitals makes $v$ a cyclic rotation of $u$. Thus $u$ and $v$ have the same multiset of entries, so $u=v$ by the elementary fact above. The final entries give $(\gamma_\ell,\alpha)^X=(\gamma_\ell,\beta)^X$, which is impossible. Thus $\alpha^\ell$ and $\beta'$ satisfy the same hypothesis for $Y_1$. Iterating this argument through all the cyclic layers, we obtain a pair with the same property for $Y$. Relabelling $Y$, its domain and this pair as $X$, $\Delta$ and $(\alpha,\beta)$, respectively, we may now treat the final symmetric layer.

Now $W=X\wr S_m$ with $m=r(X)\ge2$. Since $D(S_m)=m$, a partition $\Pi(\rho)$ with $\rho\in\mathcal R(X)^m$ is distinguishing precisely when $\rho$ lists every element of $\mathcal R(X)$ exactly once, and Proposition~\ref{bookkeeping} gives $b(W)=2$ and $r(W)=1$. The pair $(\alpha^m,\beta'')$ with $\beta''=(\alpha^{m-1},\beta)$ has the coordinate pair $(\alpha,\alpha)$, so $\alpha^m$ and $\beta''$ are nonadjacent by Lemma~\ref{basecrit}. Suppose that $z$ is adjacent to both. Then $\bigl((\alpha,z_j)^X\bigr)_j$ and $\bigl((z_j,\beta''_j)^X\bigr)_j$ list every element of $\mathcal R(X)$ exactly once, and so does $\bigl((z_j,\alpha)^X\bigr)_j$, since reversal permutes the regular orbitals. The last two tuples have the same multiset of entries and agree in the first $m-1$ entries, so they are equal by the elementary fact above. The final entries give $(z_m,\alpha)^X=(z_m,\beta)^X$, which is impossible.
\end{proof}

The next lemma verifies the hypothesis of Proposition~\ref{sprop} from a condition on $K$ involving only one of its two regular orbitals. For $\rho\in\{0,1\}^m$, write $\overline\rho$ for the vector obtained by exchanging $0$ and $1$, and let the \emph{weight} of $\rho$ be the number of entries equal to $1$.

\begin{lem}\label{wordlem}
Let $K$ be transitive on $\Delta$ with $b(K)=2$, and suppose that $\mathcal R(K)$ consists of two self-paired orbitals. Let $\alpha,\beta\in\Delta$ be distinct points admitting a labelling $\mathcal R(K)=\{\mathcal O_1,\mathcal O_2\}$ such that no point $\delta$ satisfies $(\alpha,\delta)^K=(\delta,\beta)^K=\mathcal O_1$. Let $Q\le S_m$ be transitive with $D(Q)=2$. Suppose that every $\rho$ for which $\Pi(\rho)$ is distinguishing for $Q$ has weight $m/2$ and lies in a different $Q$-orbit from $\overline\rho$. Let $M=K\wr Q$ act in product action. Then $b(M)=2$, and no common neighbour $\gamma$ of $\alpha^m$ and $\beta^m$ satisfies $(\gamma,\alpha^m)^M=(\gamma,\beta^m)^M$.
\end{lem}

\begin{proof}
Since $r(K)=2=D(Q)$, Proposition~\ref{bookkeeping} gives $b(M)=2$. Reversing an ordered base for $M$ leaves its tuple of coordinate orbitals unchanged, because the two elements of $\mathcal R(K)$ are self-paired. Every element of $\mathcal R(M)$ is therefore self-paired, so it suffices to show that $(\alpha^m,\gamma)^M\ne(\gamma,\beta^m)^M$ for every common neighbour $\gamma$. Let $w,w'\in\{0,1\}^m$ encode the tuples of coordinate orbitals of $(\alpha^m,\gamma)$ and $(\gamma,\beta^m)$, where $0$ stands for $\mathcal O_1$ and $1$ for $\mathcal O_2$. Both $\Pi(w)$ and $\Pi(w')$ are distinguishing for $Q$, so $w$ and $w'$ have weight $m/2$. If some coordinate $j$ is zero in both, then $(\alpha,\gamma_j)^K=(\gamma_j,\beta)^K=\mathcal O_1$, contrary to hypothesis. The zero sets of $w$ and $w'$ are therefore disjoint, and since each has size $m/2$ they partition $\{1,\dots,m\}$. Hence $w'=\overline w$. By hypothesis $w$ and $\overline w$ lie in different $Q$-orbits, so $(\alpha^m,\gamma)^M\ne(\gamma,\beta^m)^M$.
\end{proof}

\section{An infinite family of affine counterexamples} \label{sec:affine}

Throughout, let $V$ be a finite-dimensional vector space over $\F_p$, and let $H\le\GL(V)$ be nontrivial and irreducible. We use $\mathbf0$ for the zero vector and, when coordinates have been specified, $\mathbf1$ for the all-ones vector. Set $G=V{:}H$, acting naturally on $V$. Then $G$ is primitive, since $H$ is irreducible, and nonregular, since $G_{\mathbf0}=H\ne1$.

Let
\[
V_{\mathrm{reg}}=V_{\mathrm{reg}}(H)=\{v\in V \mid H_v=1\}
\]
be the set of vectors lying in regular $H$-orbits. For subsets $A,B\subseteq V$, write $A+B=\{a+b \mid a\in A,\ b\in B\}$, and for $k\ge1$ write $kA=A+\cdots+A$ for the $k$-fold sumset; in particular, $3A=\{a+b+c \mid a,b,c\in A\}$. Note that $V_{\mathrm{reg}}=-V_{\mathrm{reg}}$, since $H_v=H_{-v}$ for a linear group.

The next lemma gives the standard translation between the Saxl graph of $G$ and sumsets of $V_{\mathrm{reg}}$. It is proved in the same way as \cite[Lemma~2.3]{LP}, where the common neighbour property of the whole of $\Sigma(G)$ is characterised. We omit the details.

\begin{lem}\label{criterion}
Suppose that $V_{\mathrm{reg}}\ne\varnothing$. Then:
\begin{enumerate}
\item[(i)] $b(G)=2$, and two distinct vectors $u,v\in V$ are adjacent in $\Sigma(G)$ if and only if $v-u\in V_{\mathrm{reg}}$.
\item[(ii)] Translation of an ordered base $(u,v)$ to $(\mathbf0,v-u)$ identifies the regular orbitals of $G$ with the regular $H$-orbits on $V$. In particular, $r(G)$ is the number of regular $H$-orbits on $V$.
\item[(iii)] Two vertices $u,v\in V$ have a common neighbour if and only if $v-u\in2V_{\mathrm{reg}}$. In particular, $G$ has the common neighbour property if and only if $2V_{\mathrm{reg}}=V$.
\item[(iv)] If $V_{\mathrm{reg}}\cup2V_{\mathrm{reg}}\ne V$ but $3V_{\mathrm{reg}}=V$, then $\diam\Sigma(G)=3$.
\end{enumerate}
\end{lem}

The following elementary sumset observation is also used in Section~\ref{sec:searches}.

\begin{lem}\label{triplesum}
If $A\subseteq V$ satisfies $|V\setminus(A+A)|<|A|$ then $3A=V$.
\end{lem}

\begin{proof}
For $v\in V$, the translate $v-A$ has size $|A|$, so it cannot be contained in $V\setminus(A+A)$. Thus $v-A$ meets $A+A$, and hence $v\in3A$.
\end{proof}

To construct our family of counterexamples, we identify the vector
space $U_0=\F_3^9$ with $M_3(\F_3)$. Define $D\le\GL(U_0)$ to be the
subgroup of all coordinate sign changes for which every row and every
column contains an even number of sign changes. The four signs in the upper-left $2\times2$ submatrix may be chosen freely, after which the row and column parity conditions determine the remaining five signs uniquely. Since coordinate sign changes commute
and have order two, it follows that $D\cong C_2^4$. Let
$E\cong(S_3\times S_3){:}C_2$ act by row and column permutations and
by transposition. These operations preserve the defining parity
conditions, so $E$ normalises $D$, and we set
\[
L_0=D{:}E,
\qquad
G_0=U_0{:}L_0.
\]

Let
\[
u_1=\begin{pmatrix}
0&1&2\\
1&0&1\\
1&1&1
\end{pmatrix},
\qquad
u_2=\begin{pmatrix}
0&0&2\\
0&1&1\\
1&1&1
\end{pmatrix},
\qquad
O_s=u_s^{L_0}\quad(s=1,2).
\]
The accompanying GAP computation proves the following.

\begin{lem}\label{affineinput}
The group $L_0$ acts irreducibly on $U_0$, with precisely two regular orbits, namely $O_1$ and $O_2$. Moreover:
\begin{enumerate}
\item[{(i)}] $O_1+O_1=U_0$, and $|U_0\setminus(O_2+O_2)|=32$;
\item[{(ii)}] for $s,t\in\{1,2\}$, $\mathbf1\in O_s+O_t$ if and only if $s=t=1$.
\end{enumerate}
\end{lem}

Since $O_1+O_1=U_0$, we have $3O_1=U_0+O_1=U_0$. Since $O_2$ is regular, $|O_2|=|L_0|=16\cdot72=1152>32$, so Lemma~\ref{triplesum} applied with $A=O_2$ gives $3O_2=U_0$. Hence
\begin{equation}\label{eq:triple}
3O_1=3O_2=U_0.
\end{equation}

Since $2V_{\mathrm{reg}}(L_0)\supseteq O_1+O_1=U_0$, Lemma~\ref{criterion}(iii) shows that $G_0$ satisfies Conjecture~\ref{bgconj} (despite having the same order, $G_0$ is not the group in the first row of Table~\ref{tab:counterexamples}). It nevertheless fails to satisfy condition $(\dagger)$: for the pair $(\mathbf0,\mathbf1)$, Lemma~\ref{affineinput}(ii) shows that $\Sigma_{G_0}(\mathbf0)$ misses the regular $(G_0)_{\mathbf1}$-orbit $\mathbf1-O_2$. By \cite[Corollary~5.7 and Proposition~5.9]{BH}, this already yields a single counterexample, namely $G_0\wr S_2$. To obtain an infinite family of counterexamples we use a different construction. Recursively define
\begin{equation}\label{eq:tower}
U_{n+1}=U_n^3,\qquad
L_{n+1}=L_n\wr C_3,\qquad
G_{n+1}=U_{n+1}{:}L_{n+1}=G_n\wr C_3,
\end{equation}
where every wreath product is taken in product action and $C_3$ is regular on three coordinates.

\begin{thm}\label{thm:affine}
For every $n\ge0$, the group $G_n$ is a primitive affine group of degree $3^{9\cdot3^n}$ with $b(G_n)=r(G_n)=2$. If $n\ge1$, the vertices $\mathbf0$ and $\mathbf1$ of $\Sigma(G_n)$ are nonadjacent and have no common neighbour, and $\diam\Sigma(G_n)=3$. In particular, the groups $G_n$ for $n\ge1$ form an infinite family of affine counterexamples to Conjecture~\ref{bgconj}.
\end{thm}

\begin{proof}
Since $\dim_{\F_3}U_0=9$ and $\dim U_{n+1}=3\dim U_n$ by \eqref{eq:tower}, the degree of $G_n$ is $|U_n|=3^{9\cdot3^n}$. Lemma~\ref{affineinput} shows that $L_0$ is irreducible, so $G_0$ is primitive. Since $L_0\ne1$ and $L_{n+1}=L_n\wr C_3$, every $L_n$ is nontrivial, so $G_n=U_n{:}L_n$ is nonregular. Thus, if $G_n$ is primitive, then Lemma~\ref{paprim} shows that $G_{n+1}=G_n\wr C_3$ is primitive, since $C_3$ is transitive on three points. Moreover, $U_n$ is a regular elementary abelian normal subgroup of $G_n$, so $G_n$ is of affine type. Lemmas~\ref{criterion} and~\ref{affineinput} give $b(G_0)=r(G_0)=2$. Suppose that $b(G_n)=r(G_n)=2$. Since $D(C_3)=2$, Proposition~\ref{bookkeeping} gives $b(G_{n+1})=2$ and
\[
r(G_{n+1})=\frac{r(G_n)^3-r(G_n)}{3}=\frac{2^3-2}{3}=2.
\]

Now suppose that $\gamma$ is a common neighbour of $\mathbf0$ and $\mathbf1$ in $\Sigma(G_0)$. By Lemma~\ref{criterion}, there are $s,t\in\{1,2\}$ such that $\gamma\in O_s$ and $\mathbf1-\gamma\in O_t$. Since $\mathbf1=\gamma+(\mathbf1-\gamma)$, Lemma~\ref{affineinput}(ii) gives $s=t=1$. It follows that
\[
\Sigma_{G_0}(\mathbf0)\cap\Sigma_{G_0}(\mathbf1)\subseteq\mathbf1-O_1.
\]
Since scalar multiplication by $-1$ centralises $L_0$, the set $-O_1$ is again a single regular $L_0$-orbit, so $\mathbf1-O_1$ is a single regular $(G_0)_{\mathbf1}$-orbit. The vertices adjacent to $\mathbf1$ are precisely those lying in regular $(G_0)_{\mathbf1}$-orbits, so $\Sigma_{G_0}(\mathbf0)$ meets at most one such orbit and the pair $(\mathbf0,\mathbf1)$ satisfies the hypothesis of Proposition~\ref{cprop}. That proposition shows that the corresponding constant all-zero and all-ones vectors in $U_1$ are nonadjacent and have no common neighbour. These two properties pass to constant lifts from $G_n$ to $G_{n+1}$: nonadjacency follows from Lemma~\ref{basecrit}, and every coordinate of a common neighbour of the two lifts would be a common neighbour of $\mathbf0$ and $\mathbf1$ in $\Sigma(G_n)$. Thus the corresponding vectors in $U_n$ are nonadjacent and have no common neighbour for every $n\ge1$. Hence $\diam\Sigma(G_n)\ge3$ for $n\ge1$.

For the reverse inequality, let $\mathcal O_i$ be the regular orbital of $G_0$ corresponding to $O_i$, where $i\in\{1,2\}$. Given $x,y\in U_0$, \eqref{eq:triple} gives $a,b,c\in O_i$ such that $y-x=a+b+c$. Hence $(x,x+a),(x+a,x+a+b),(x+a+b,y)\in\mathcal O_i$. Repeated application of Lemma~\ref{inherit} shows that any two vertices of $\Sigma(G_n)$ are joined by a path of length at most three.
\end{proof}

We conclude this section with one further affine counterexample, whose point stabiliser is perfect.

\begin{prop}\label{prop:perfect}
Let $V$ be the deleted permutation module over $\F_3$ for the natural action of $H=\AGL_4(2)$ on $16$ points. Then $G=V{:}H$ is primitive of degree $3^{15}$, with $b(G)=2$ and $\diam\Sigma(G)=3$. 
\end{prop}

\begin{proof}
The accompanying computation \cite{AlunaAdam} verifies that $H$ acts faithfully and irreducibly on $V$, that the regular vectors form a single orbit, and that $|V\setminus(V_{\mathrm{reg}} \cup 2V_{\mathrm{reg}})|=32$. Thus $G$ is primitive of degree $3^{15}$ and $|V_{\mathrm{reg}}|=|H|$. Moreover, $V_{\mathrm{reg}}\subseteq2V_{\mathrm{reg}}$, since $v=2v+2v$ and $2v\in V_{\mathrm{reg}}$ whenever $v\in V_{\mathrm{reg}}$. Therefore, Lemmas~\ref{criterion} and~\ref{triplesum} give $b(G)=2$ and $\diam\Sigma(G)=3$.
\end{proof}

\section{An infinite family of product type counterexamples} \label{sec:product}

Let $K=\operatorname{PSL}_2(11)$ act primitively on the $55$ cosets of a dihedral subgroup $J$ of order $12$. By \cite[Table~4]{BH}, this action has $b(K)=r(K)=2$. Let
\[
Q=\{(a,b)\in S_4\times S_4 \mid  ab^{-1}\in A_4\}
\cong(A_4\times A_4){:}C_2.
\]
Via the action on $A_4$ given by $x\mapsto axb^{-1}$, this is a faithful transitive soluble permutation group of degree $12$. The accompanying GAP computation proves the following.

\begin{lem}\label{productinput}
The two elements of $\mathcal R(K)$ are self-paired. Moreover:
\begin{enumerate}
\item[{(i)}] there are distinct points $\alpha,\beta$, and a labelling of $\mathcal R(K)$ as $\{\mathcal O_1,\mathcal O_2\}$, such that there is no point $\delta$ with
\[
(\alpha,\delta)^K=(\delta,\beta)^K=\mathcal O_1;
\]
\item[{(ii)}] for each $\mathcal O\in\mathcal R(K)$ and all points $x,y$, there are points $u,v$ with $(x,u),(u,v),(v,y)\in\mathcal O$;
\item[{(iii)}] the binary vectors $\rho\in\{0,1\}^{12}$ for which $\Pi(\rho)$ is distinguishing for $Q$ form exactly two $Q$-orbits; every such vector has weight six, and $\rho$ and $\overline\rho$ lie in different orbits.
\end{enumerate}
\end{lem}

Define
\begin{equation}\label{eq:patower}
M_0=K\wr Q,\qquad
M_{n+1}=M_n\wr C_3,\qquad
G_n=M_n\wr C_2\quad(n\ge0),
\end{equation}
where the cyclic groups act regularly and every wreath product acts in product action.

\begin{thm} \label{thm:product}
For every $n\ge0$, the group $G_n$ is a primitive group of product type having degree $55^{\,24\cdot3^n}$ and soluble point stabilisers, with $b(G_n)=2$. Furthermore, $r(G_n)=1$, so the group $G_n$ is transitive on the arcs of $\Sigma(G_n)$. The graph $\Sigma(G_n)$ contains two nonadjacent vertices with no common neighbour, and $\diam\Sigma(G_n)=3$. In particular, the groups $G_n$ form an infinite family of product type counterexamples to Conjecture~\ref{bgconj}.
\end{thm}

\begin{proof}
With $Q_0=Q$, let $Q_{n+1}=Q_n\wr C_3$ and $P_n=Q_n\wr C_2$, both in imprimitive action; then $P_n$ is transitive of degree $24\cdot3^n$. Since $(K\wr Q)\wr R\cong K\wr(Q\wr R)$ in product action, with $Q\wr R$ acting imprimitively, iterating \eqref{eq:patower} shows that $G_n$ is permutation isomorphic to $K\wr P_n$ in product action. Since $K$ is primitive and nonregular and $P_n$ is transitive, Lemma~\ref{paprim} shows that $G_n$ is primitive. Its degree is $55^{\,24\cdot3^n}$. Its socle is $K^{\,24\cdot3^n}$ with $K$ simple, so $G_n$ is of product type (see \cite{LPS} or \cite[Section~2.2]{BH}). Its point stabiliser is $J^{\,24\cdot3^n}{:}P_n$. This is soluble, since $J$ is dihedral and hence soluble, and $P_n$ is an iterated wreath product of the soluble groups $Q$, $C_3$ and $C_2$.

Since $r(K)=2$ we may identify $\mathcal R(K)^{12}$ with $\{0,1\}^{12}$, so by Lemma~\ref{productinput}(iii) the distinguishing tuples form two $Q$-orbits. Hence $D(Q)=2$, and Proposition~\ref{bookkeeping} gives $b(M_0)=r(M_0)=2$; since $D(C_3)=2$, a simultaneous induction using Proposition~\ref{bookkeeping} and \eqref{eq:cyc} gives $b(M_n)=r(M_n)=2$ for all $n$. The two regular orbitals of $K$ are self-paired, so Lemma~\ref{wordlem}, applied with $m=12$, shows that no common neighbour $\gamma$ of $\alpha^{12}$ and $\beta^{12}$ in $\Sigma(M_0)$ satisfies $(\gamma,\alpha^{12})^{M_0}=(\gamma,\beta^{12})^{M_0}$. Since $r(M_n)=2$, the last layer of \eqref{eq:patower} reads $G_n=M_n\wr S_{r(M_n)}$. Proposition~\ref{sprop} therefore applies to $M_0$, with the $n$ wreath products by $C_3$ in between, and gives $b(G_n)=2$, $r(G_n)=1$, and two nonadjacent vertices of $\Sigma(G_n)$ with no common neighbour. Hence $\diam\Sigma(G_n)\ge3$.

Finally, Lemma~\ref{productinput}(ii) supplies the hypothesis of Lemma~\ref{inherit} for $K$. Applying Lemma~\ref{inherit} successively through the layers of \eqref{eq:patower} gives $\diam\Sigma(G_n)\le 3$.
\end{proof}

\section{An infinite family of twisted wreath type counterexamples} \label{sec:tw}

Let $P=S_6$ act naturally on $\{1,\dots,6\}$, let $E=P_{\{1\}}\cong S_5$, and let $T=\operatorname{soc}(E)\cong A_5$. Write $\varphi:E\to\operatorname{Aut}(T)$ for the conjugation action. By \cite[Example~9.3]{F}, the twisted wreath product
\[
H=T\operatorname{twr}_{\varphi}P=N_0{:}P
\]
is primitive of twisted wreath type, with regular socle $N_0\cong T^6$ and degree $60^6$.

We use the coordinate model of \cite[Section~4]{F}, with all actions on the right. Let $a_1=\operatorname{id}$ and $a_i=(1\,i)$ for $2\le i\le6$, so that $\{a_1,\dots,a_6\}$ is a left transversal for $E$ in $P$ satisfying $i^{a_i}=1$; evaluation on it identifies $N_0$ with $T^6$, where $T$ acts on $\{2,\dots,6\}$. Every element of $H$ is uniquely $pn$ with $p\in P$ and $n\in N_0$, and acts on $N_0$ by $z\mapsto z^pn$.

The computational input for this section is the number of regular suborbits of $H$, together with a pair of points to which Proposition~\ref{sprop} applies.

\begin{lem}\label{twinput}
We have $b(H)=2$ and $r(H)=64\,790\,243$. Furthermore, there exists an element $w\in N_0$ such that the pair $(1,w)$ satisfies the hypothesis of Proposition~\ref{sprop}.
\end{lem}

\begin{proof}
The value of $r(H)$ and the finite check below are verified by \cite{AlunaAdam}. The stabiliser of the point $1$ is $P$, so no single point is a base, while $r(H)\ge1$ supplies a base of size two; hence $b(H)=2$. Let
\[
w=\bigl((2\,3\,6),\,(2\,4\,5\,3\,6),\,(2\,6\,5),\,
(2\,3\,5\,4\,6),\,(3\,4\,5),\,(4\,5\,6)\bigr)\in N_0.\]
The pair $(1,w)$ satisfies the hypothesis of Proposition~\ref{sprop} provided no element of $H$ carries $(\gamma,1)$ to $(\gamma,w)$ for any $\gamma\in N_0$. Write such an element of $H$ as $pn$ with $p\in P$ and $n\in N_0$. Since $1^{pn}=n$, the second coordinate forces $n=w$, and the first is then fixed precisely when $\gamma^pw=\gamma$. So it suffices to show that $z^pw\ne z$ for all $z\in N_0$ and $p\in P$.

For $p\in P$ and $1\le i\le6$, set $e_{p,i}=a_i^{-1}pa_{i^p}\in E$. Then $z^pw=z$ is equivalent to $z_{i^p}=z_i^{e_{p,i}}w_{i^p}$ for every $i$. Along a cycle of $p$, one entry determines the rest, subject to a single condition when the cycle closes. As there are only $|T|=60$ choices for the initial entry, each cycle can be tested directly, and for every $p\in P$ some cycle admits no consistent value.
\end{proof}

The group $H$ is not itself a counterexample. Indeed, the neighbours of $1$ in $\Sigma(H)$ are the points lying in regular $P$-orbits. Write $S$ for the union of these orbits, so that $|S|=|P|\,r(H)$. Since $N_0$ acts on itself by right multiplication, the neighbours of $x\in N_0$ are the elements of $Sx$. As $|S|>\tfrac12|N_0|$, the sets $Sx$ and $Sy$ meet for all $x,y\in N_0$, so $H$ has the common neighbour property.

Write $c_0=r(H)$ and $c_{n+1}=(c_n^2-c_n)/2$. Recursively define
\begin{equation}\label{eq:twtower}
H_0=H,\qquad
H_{n+1}=H_n\wr C_2,\qquad
G_n=H_n\wr S_{c_n},
\end{equation}
where every wreath product is taken in product action and $C_2$ is regular on two coordinates.

\begin{thm}\label{twfamilythm}
For every $n\ge0$, the group $G_n$ is a primitive group of twisted wreath type, of degree $60^{\,6\cdot2^nc_n}$, with $b(G_n)=2$ and $r(G_n)=1$. The graph $\Sigma(G_n)$ contains two nonadjacent vertices with no common neighbour, so $\diam\Sigma(G_n)\ge3$. In particular, the groups $G_n$ form an infinite family of twisted wreath type counterexamples to Conjecture~\ref{bgconj}.
\end{thm}

\begin{proof}
Applying \cite[Lemma~4.9]{F} to each of the wreath products in \eqref{eq:twtower} shows that every $H_n$ and every $G_n$ is primitive of twisted wreath type, and that $\operatorname{soc}(G_n)=N_0^{\,2^nc_n}\cong A_5^{\,6\cdot2^nc_n}$. As the socle of a group of twisted wreath type is regular, the degree of $G_n$ is the order of its socle, namely $60^{\,6\cdot2^nc_n}$.

By Lemma~\ref{twinput}, $b(H_0)=2$ and $r(H_0)=c_0\ge3$. Since $D(C_2)=2$, Proposition~\ref{bookkeeping} and \eqref{eq:cyc} show inductively that $b(H_n)=2$ and $r(H_{n+1})=(r(H_n)^2-r(H_n))/2=c_{n+1}$, so $r(H_n)=c_n\ge3$ for all $n$. Lemma~\ref{twinput} also supplies a pair satisfying the hypothesis of Proposition~\ref{sprop}. Applying that proposition to $H_0$, with the $n$ wreath products by $C_2$ and then $S_{c_n}$, gives $b(G_n)=2$, $r(G_n)=1$, and two nonadjacent vertices of $\Sigma(G_n)$ with no common neighbour. Hence $\diam\Sigma(G_n)\ge3$. Finally, $c_0>3$, so $c_{n+1}=c_n(c_n-1)/2>c_n$ and the degrees are strictly increasing.
\end{proof}

In contrast with the previous sections we make no exact diameter claim. Lemma~\ref{inherit} cannot help here. A regular orbital of $H$ has out-valency $|P|$, so if $\mathcal O\in\mathcal R(H)$ and $x\in N_0$, then at most $|P|^3<|N_0|$ points $y$ can satisfy $(x,u),(u,v),(v,y)\in\mathcal O$ for some $u,v\in N_0$. Thus the hypothesis of that lemma fails for reasons of size alone. We do not know the diameters of the graphs in Theorem~\ref{twfamilythm}.

\section{Counterexamples at every base size} \label{sec:everybase}

Let $d\ge3$ be odd, let $q=3^d$, and let $C$ be the subgroup of squares in $\F_q^\times$. Since $q\equiv3\pmod4$, the group $C$ has odd order and $-1\notin C$. Let the resulting Frobenius group
\[
H_q=\{x\mapsto ax+b \mid a\in C,\ b\in\F_q\}
\]
act on $\Omega_q=\F_q$, and let
\[
V_q=\Bigl\{v\in\F_2^{\Omega_q} \mid \sum_{\omega\in\Omega_q}v_\omega=0\Bigr\}
\]
be the deleted permutation module, of dimension $q-1$. Write $L_q=V_q{:}H_q$ for the resulting affine group in its natural action on $V_q$, and for $j\ge1$ let $r_j$ be the number of regular $H_q$-orbits on $V_q^j$.

\begin{lem}\label{genmodule}
The group $H_q$ acts faithfully and irreducibly on $V_q$, with at least two regular orbits. Moreover, for every $j\ge1$:
\begin{enumerate}
\item[{(i)}] $r_{j+1}\ge|V_q|\,r_j$, and in particular $r_j<r_{j+1}$;
\item[{(ii)}] $L_q$ has exactly $r_j$ regular orbits on $V_q^{j+1}$, matched with the regular $H_q$-orbits on $V_q^{j}$ by deleting a leading $\mathbf0$.
\end{enumerate}
\end{lem}

\begin{proof}
Since $q$ is odd, $\F_2^{\Omega_q}=V_q\oplus\langle\mathbf1\rangle$. The second summand is the trivial $H_q$-module, so the faithful action on $\F_2^{\Omega_q}$ restricts faithfully to $V_q$. Let $N=(\F_q,+)$ be the translation subgroup. The characters of $N$ take values in the cube roots of unity, so over $\F_4$ the restriction of $V_q$ to $N$ is the sum of the $q-1$ nontrivial character spaces, indexed by $u\in\F_q^\times$. The Galois pairs $\{u,-u\}$ give the pairwise nonisomorphic irreducible $\F_2N$-constituents. Each pair contains a unique square since $-1\notin C$, so $C$ acts regularly on these pairs. As $|N|$ is odd, the restriction to $N$ is semisimple and multiplicity-free. By Clifford's theorem, an $H_q$-submodule of $V_q$ is a sum of a $C$-invariant set of constituents. The transitivity of $C$ then leaves only $0$ and $V_q$.

A nonzero translation has $q/3$ orbits on $\Omega_q$, and an element with nontrivial multiplier has one fixed point together with orbits of length at least $3$. So every $h\ne1$ in $H_q$ has at most $(q+2)/3$ orbits, and hence fixes at most $2^{(q+2)/3}$ vectors. Therefore at most $|H_q|\,2^{(q+2)/3}$ vectors of $V_q$ are nonregular. As $q\ge27$ this is less than $|V_q|-|H_q|$, so $r_1\ge2$.

For (i), append each vector of $V_q$ in turn to a representative of a regular $H_q$-orbit on $V_q^j$. Each resulting tuple has trivial stabiliser. Two of them lie in the same orbit only if the appended vectors agree, and tuples extending distinct orbits lie in distinct orbits. For (ii), every $L_q$-orbit on $V_q^{j+1}$ contains a tuple with first entry $\mathbf0$. Two such tuples lie in the same $L_q$-orbit if and only if their final $j$ entries lie in the same $H_q$-orbit. Since $(L_q)_{\mathbf0}=H_q$, the stabiliser of such a tuple is the stabiliser in $H_q$ of its last $j$ entries. Hence one orbit is regular if and only if the other is.
\end{proof}

Fix $B\ge2$, let $t=B-1$ and $r=r_t$, and define
\begin{equation}\label{eq:genconstruction}
G_{B,d}=L_q\wr S_r
       =V_q^{\,r}{:}(H_q\wr S_r)
\end{equation}
in product action on $W=V_q^{\,r}$.

\begin{lem}\label{genbase}
We have $b(G_{B,d})=B$.
\end{lem}

\begin{proof}
By Lemma~\ref{genmodule}, $L_q$ is primitive and nonregular and $r=r_t\ge2$, so Lemma~\ref{paprim} shows that $G_{B,d}$ is primitive in product action. By Lemma~\ref{genmodule}(ii) the number of regular $L_q$-orbits on $V_q^{k}$ is $r_{k-1}$ for $k\ge2$, and the $r_j$ are strictly increasing by Lemma~\ref{genmodule}(i). Since $D(S_r)=r=r_t$, the least $k\ge2$ with $r_{k-1}\ge r$ is $k=t+1$. As $G_{B,d}$ is nonregular we have $b(G_{B,d})\ge2$, so \cite[Theorem~6.1]{FHLR} gives $b(G_{B,d})=t+1=B$.
\end{proof}

Let $M=H_q\wr S_r$ be the stabiliser of $\mathbf0$ in $G_{B,d}$, and let $A\subseteq W$ be the set of vectors occurring in bases of size $t$ for $M$. A base of size $B$ for $G_{B,d}$ containing $\mathbf0$ consists of $\mathbf0$ together with a base of size $t$ for $M$. Since $b(G_{B,d})=B$ by Lemma~\ref{genbase}, the set $A$ is the neighbourhood of $\mathbf0$ in the generalised Saxl graph. By \cite[Lemma~5.1]{FHLR}, for $v\in W\setminus\{\mathbf0\}$ the vertices $\mathbf0$ and $v$ have a common neighbour if and only if $v\in A+A$.

Fix $\omega_0\in\Omega_q$, and let $u\in V_q$ be the characteristic vector of $\Omega_q\setminus\{\omega_0\}$. Thus $u_{\omega_0}=0$ and $u_\omega=1$ for $\omega\ne\omega_0$. Set $z=(u,\mathbf0,\dots,\mathbf0)\in W$.

\begin{lem}\label{genobstruction}
We have $z\notin A\cup(A+A)$.
\end{lem}

\begin{proof}
Let $x\in A$ and choose a base of size $t$ for $M$ containing $x$. Together with $\mathbf0$ this is a base of size $B$ for $G_{B,d}$, which we display as a $B\times r$ array whose first row is $\mathbf0$ and whose second is $x$.

By \cite[Lemma~6.2]{FHLR}, the columns represent all $r$ regular $L_q$-orbits on $V_q^B$ exactly once. Each column begins with $\mathbf0$, so by Lemma~\ref{genmodule}(ii) deleting that entry matches its orbit with a regular $H_q$-orbit on $V_q^t$, and every such orbit arises exactly once. Now $x_i$ is the first entry of the $i$th truncated column, and a regular $H_q$-orbit on $V_q^t$ determines the $H_q$-orbit of its first entry. Hence, for every $H_q$-orbit $O$ on $V_q$, the number of coordinates $i$ with $x_i\in O$ does not depend on $x$.

If $z\in A$, then as $z$ has $r-1$ coordinates equal to $\mathbf0$, so does every element of $A$. The rows of the array other than the first all lie in $A$, so their nonzero entries occupy at most $t$ columns. Since $t\ge1$, $r_1\ge2$, and the $r_j$ are strictly increasing by Lemma~\ref{genmodule}, we have $r=r_t\ge t+1$, so some column has all zero entries. Deleting its leading $\mathbf0$ leaves the zero tuple of $V_q^t$, whose $H_q$-orbit is not regular. This is a contradiction, so $z\notin A$.

Suppose that $z=x+y$ with $x,y\in A$. Then $x_i=y_i$ for every $i>1$, while $y_1=x_1+u$. Since $x$ and $y$ have the same number of coordinates in each $H_q$-orbit, cancelling the common orbit labels in coordinates $i>1$ shows that $x_1$ and $y_1$ lie in the same $H_q$-orbit, say $y_1=x_1^h$ with $h\in H_q$. Thus $u=x_1+y_1=(1+h)x_1$ lies in $(1+h)V_q$.

For an orbit $\Lambda$ of $\langle h\rangle$ on $\Omega_q$, let $\lambda_\Lambda(v)=\sum_{\omega\in\Lambda}v_\omega$. Since $h$ permutes $\Lambda$, the functional $\lambda_\Lambda$ annihilates $(1+h)V_q$. As $H_q$ has odd order, every $\langle h\rangle$-orbit has odd length. Moreover $h$ has at least two orbits: the identity has $q$, a nonzero translation has $q/3$, and an element with nontrivial multiplier has a unique fixed point together with at least one further orbit. Choose an orbit $\Lambda$ avoiding $\omega_0$. Then
\[
\lambda_\Lambda(u)=|\Lambda|\equiv1\pmod2,
\]
contradicting $u\in(1+h)V_q$.
\end{proof}

\begin{thm}\label{genthm}
For every $B\ge2$ and every odd $d\ge3$, the group $G_{B,d}$ is a primitive affine group of degree $2^{(q-1)r_t}$ with $b(G_{B,d})=B$, where $q=3^d$ and $t=B-1$. The vertices $\mathbf0$ and $z$ of its generalised Saxl graph are nonadjacent and have no common neighbour. In particular, for each fixed $B$ the groups $G_{B,d}$ form an infinite family of counterexamples to \cite[Conjecture~1.2]{FHLR}.
\end{thm}

\begin{proof}
By Lemma~\ref{genmodule}, $H_q$ acts faithfully and irreducibly on $V_q$, so $L_q$ is primitive and nonregular. Since $t\ge1$, Lemma~\ref{genmodule} gives $r=r_t\ge r_1\ge2$, and $S_r$ is transitive, so Lemma~\ref{paprim} shows that $G_{B,d}=L_q\wr S_r$ is primitive in product action. The decomposition in \eqref{eq:genconstruction} exhibits $V_q^{\,r}$ as a regular elementary abelian normal subgroup, so $G_{B,d}$ is affine, of degree $|V_q|^{\,r}=2^{(q-1)r_t}$.

Lemma~\ref{genbase} gives $b(G_{B,d})=B$. By Lemma~\ref{genobstruction}, $z\notin A\cup(A+A)$, so $\mathbf0$ and $z$ are nonadjacent and have no common neighbour. Finally, $2^{(q-1)r_t}\ge2^{q-1}$ is unbounded as $d$ ranges over the odd integers at least $3$, so for fixed $B$ infinitely many of the $G_{B,d}$ are pairwise nonisomorphic.
\end{proof}

\section{Computational results}\label{sec:searches}

The uniform construction of Section~\ref{sec:everybase} is independent of the searches in this section. The other counterexamples in this paper emerged from a systematic search. Table~\ref{tab:ranges} summarises the ranges that have been exhausted, and Table~\ref{tab:counterexamples} the counterexamples they contain.

Each of the four counterexamples is an affine group $V{:}H$ whose point stabiliser $H$ is a monomial subgroup of $\GL(V)$. In each case the regular vectors form a single orbit and $V_{\mathrm{reg}}\subseteq2V_{\mathrm{reg}}$. The column $|V\setminus2V_{\mathrm{reg}}|$ of Table~\ref{tab:counterexamples} therefore also counts the vectors outside $V_{\mathrm{reg}}\cup2V_{\mathrm{reg}}$. This number is positive and less than $|H|$, so Lemmas~\ref{criterion} and~\ref{triplesum} give $\diam\Sigma(V{:}H)=3$. Here $|D_{18}|=18$ and $|12T35|=|12T38|=72$ with notation as in \cite{GAP}.

\begin{table}[!htbp]
\centering
\begin{tabular}{llc}
\hline
primitive groups & degrees exhausted & counterexamples\\
\hline
all & $n<8192$ & $0$\\
affine, $H$ soluble & $n < 2^{24}$ & $3$\\
affine, $H$ insoluble & $n < 2^{18}$ & $1$\\
affine, $H$ almost quasisimple of sporadic type & all $n$ & $0$\\
non-affine & $n\le10^{8}$ & $0$\\
diagonal & $n\le10^{24}$ & $0$\\
\hline
\end{tabular}
\smallskip
\caption{The exhausted ranges.}
\label{tab:ranges}
\end{table}

\begin{table}[!htbp]
\centering
\begin{tabular}{llrrc}
\hline
$V$ & $H$ & $|V_{\mathrm{reg}}|$ &
$|V\setminus2V_{\mathrm{reg}}|$ & $\diam\Sigma$\\
\hline
$\F_3^{9}$  & $C_2^6{:}D_{18}$    & $1152$    & $96$   & $3$\\
$\F_3^{10}$ & $C_2^5{:}S_5$       & $3840$    & $64$   & $3$\\
$\F_3^{12}$ & $C_2^8{:}(12T35)$   & $18\,432$ & $1600$ & $3$\\
$\F_3^{12}$ & $C_2^8{:}(12T38)$   & $18\,432$ & $1600$ & $3$\\
\hline
\end{tabular}
\smallskip
\caption{The four counterexamples found in the affine searches; see \cite{AlunaAdam}.}
\label{tab:counterexamples}
\end{table}

\subsection{Methods}

Groups are constructed in GAP \cite{GAP} and Magma \cite{Magma}; the graph computations use our own C++ code. Source code for the computations described in this section is available in \cite{AlunaAdam}.

Three observations restrict which groups can occur as counterexamples.
\begin{enumerate}
\item[{(a)}] A base-two group of degree $n$ has order at most $n(n-1)$; an affine stabiliser with a regular orbit has $|H|\le|V|-1$.
\item[{(b)}] If $\Sigma(G)$ has valency greater than $n/2$ then any two neighbourhoods meet, so $G$ has the common neighbour property; for affine groups, $|V_{\mathrm{reg}}|>|V|/2$ forces $2V_{\mathrm{reg}}=V$.
\item[{(c)}] If $G\le N\le\Symgrp(\Omega)$ with $b(G)=b(N)=2$, then $\Sigma(N)$ is a spanning subgraph of $\Sigma(G)$, so the common neighbour property passes from $N$ to $G$; for affine groups $H\le K\le\GL(V)$ gives $V_{\mathrm{reg}}(K)\subseteq V_{\mathrm{reg}}(H)$, so $2V_{\mathrm{reg}}(K)=V$ forces $2V_{\mathrm{reg}}(H)=V$.
\end{enumerate}
By (c), one overgroup with the common neighbour property settles every base-two subgroup below it, so the search descends only through maximal candidates and few groups reach a graph computation.

The primitive groups of degree less than $8192$ are classified in \cite{CQRD} and \cite{St}, and the resulting library is available in GAP and Magma. Burness and Giudici verified their conjecture below degree $4096$ \cite[Section~4]{BG}. We reproduced that computation independently and extended it to the whole library, finding no counterexample.

\subsection{Affine groups}

Given generators for an irreducible $H\le\GL_d(p)$, we enumerate the $H$-orbits on $V$ and test $2V_{\mathrm{reg}}=V$. When $2V_{\mathrm{reg}}\ne V$, a breadth-first search computes the diameter of $\Sigma(G)$.

For soluble stabilisers, H\"ofling's IRREDSOL library \cite{Hof} contains, up to conjugacy, every irreducible soluble subgroup of $\GL_d(q)$ with $q^d\le2^{24}-1$. Using (c), we exhausted the library in the remaining range $8192\le q^d\le2^{24}-1$, finding precisely the three soluble examples in Table~\ref{tab:counterexamples}.

For insoluble stabilisers in dimensions $d=2$ and $3$, Magma's function \texttt{IrreducibleSubgroups} returns a complete set of representatives of the conjugacy classes of insoluble irreducible subgroups of $\GL_d(p)$. For $d\ge4$ we descend from the maximal insoluble irreducible subgroups, for which Magma's \texttt{ClassicalMaximals} provides complete lists when $d\le17$. Since $p^d<2^{18}$ in our range, this covers every dimension that occurs.

\begin{prop}\label{prop:sporadic}
Conjecture~\ref{bgconj} holds for every primitive affine group whose point stabiliser is almost quasisimple of sporadic type.
\end{prop}

\begin{proof}
Write the group as $V{:}H$, where $V=\F_q^d$ and $H\le\GL_d(q)$. Theorem~1.2 of \cite{LP} establishes this apart from the ten rows of \cite[Table~1]{LP}, some of which represent several groups, according to the choice of a scalar subgroup $Z\le H$. For those rows the computations in \cite{AlunaAdam} either handle each $Z$ in turn, or handle the largest such $Z$ and then apply observation (c) above. For every base-two group arising in six of the rows, the computations verify that $2V_{\mathrm{reg}}=V$. For every group arising in the remaining four rows, $q>2$ and the computations verify that $|V_{\mathrm{reg}}|>q^{d-1}-1$. Since scalar matrices are central in $\GL_d(q)$, the set $V_{\mathrm{reg}}$ is closed under multiplication by $\F_q^\times$. Moreover $\mathbf0\notin V_{\mathrm{reg}}$, as $H\ne1$. An immediate consequence of \cite[Theorem~2]{DMP} is that $A+A=\F_q^d$ whenever $A\subseteq\F_q^d\setminus\{\mathbf0\}$ is closed under multiplication by $\F_q^\times$ with $|A|>q^{d-1}-1$ and $q > 2$. Applying this with $A=V_{\mathrm{reg}}$ gives $2V_{\mathrm{reg}}=V$ in these cases too, and Lemma~\ref{criterion}(iii) gives the common neighbour property in every case.
\end{proof}

\subsection{Non-affine groups}

Above degree $8191$ there is no library of non-affine primitive groups. Following \cite{St} we find the non-affine primitive groups of degree at most $10^8$ with base size two.

For a diagonal type group $G$ with $\operatorname{soc}(G)=T^k$, we use Huang's classification of base-two diagonal type groups \cite[Theorem~1]{Hu}. If the top group is neither $A_k$ nor $S_k$, then \cite[Theorem~5.6]{HuT} gives the common neighbour property. If it is $A_k$ or $S_k$, then $b(G)=2$ forces $2<k<|T|$ by \cite[Corollary~2.4]{Hu}, and since $|T|^{k-1}\le 10^{24}$ and $|T|\ge60$ we have $k\le14 \le |T|-3$. The full normaliser $N=T^k.(\operatorname{Out}(T)\times S_k)$ has $b(N)=2$ by \cite[Theorem~1(ii)]{Hu}. In all cases we check that $N$ has the common neighbour property, which implies the result for $G$. To do this, in most cases we prove that any vertex of $\Sigma(N)$ is adjacent to more than half of the other vertices. That leaves the seven cases with $T=A_5$ and $8\le k\le14$. For these, we combine Huang's characterisation of bases \cite[Lemmas~2.15 and~2.16]{Hu} with a counting argument on $k$-subsets of $A_5$ to prove that any two vertices have a common neighbour.

\pagebreak

\section*{Acknowledgements}
The authors would like to thank the Isaac Newton Institute for Mathematical Sciences, Cambridge, for support and hospitality during the programme Algebraic groups, geometry, invariants and related topics, where work on this paper was undertaken. Rizzoli also acknowledges support from the Additional Funding Programme for Mathematical Sciences, delivered by EPSRC (EP/V521917/1), and the Heilbronn Institute for Mathematical Research.

For the purpose of open access, the authors have applied a Creative Commons Attribution (CC BY) licence to any Author Accepted Manuscript version arising from this submission.

\section*{Declaration of generative AI and AI-assisted technologies}
This project began when we set out, with support from Codex, to prove the common neighbour conjecture for soluble affine groups and formalise the proof in Lean. A computational search of the IRREDSOL library, which we ran as a falsification test, instead produced counterexamples, and we redirected the project towards constructing, understanding and generalising them. We then worked with Codex, ChatGPT Pro and Claude to discover further examples and constructions, search the literature, and draft and revise the manuscript.

We chose the libraries, designed the search strategies and specified the key techniques. We diagnosed bottlenecks and redirected or terminated unpromising searches. Codex implemented and debugged much of the Magma, GAP, Python and C++ code. The Lean formalisation was produced primarily by Codex; we reviewed its theorem statements and verified the completed formalisation. 

We reviewed all AI-assisted output, verified every proof and computation reported here, and take full responsibility for the content.

\end{document}